\documentclass[12pt,reqno]{amsart}

\usepackage{amscd,amsmath}
\usepackage[all]{xy}
\usepackage{graphicx}

\usepackage[british]{babel}
\usepackage{enumerate,amssymb}
\usepackage{ucs}
\usepackage[utf8x]{inputenc}
\usepackage[T1]{fontenc}

\usepackage{geometry}
\newcommand{\Q}{\ensuremath{\mathbb{Q}}}

\newcommand{\Z}{\ensuremath{\mathbb{Z}}}

\DeclareMathOperator{\card}{card}
\DeclareMathOperator{\im}{im}

\newtheorem{teo}{Theorem}
\newtheorem{teorema}{Theorem}[section]
\newtheorem{lema}[teorema]{Lemma} %
\newtheorem{cor}[teorema]{Corollary} 
\newtheorem{prop}[teorema]{Proposition}
 
\theoremstyle{definition}

\theoremstyle{remark} 
\newtheorem{obs}[teorema]{Observation}
 
\begin{document}


\title[The $R_{\infty}$ property for nilpotent quotients of GSBS groups]{The $R_{\infty}$ property for nilpotent quotients of Generalized Solvable Baumslag-Solitar groups}


\author{Wagner C. Sgobbi} 
\address[Wagner C. Sgobbi]{Departamento de Matem\'atica - Instituto de Matem\'atica e Estat\'istica - Universidade de S\~ao Paulo, Caixa Postal 66.281 - CEP 05314-970, S\~ao Paulo - SP, Brazil}
\thanks{Wagner C. Sgobbi was supported by grants 2017/21208-0 and 2019/03150-0, São Paulo Research Foundation (FAPESP) }
\email{wagnersgobbi@dm.ufscar.br }

\author{ Dalton C. Silva} 
\address[Dalton C. Silva]{Instituto Federal de Educa\c c\~ao, Ci\^encia e Tecnologia de S\~ao Paulo, 11665-071, Caraguatatuba, SP, Brazil}
\email{dalton.couto@ifsp.edu.br }

\author{Daniel Vendr\'uscolo}
 \address[Daniel Vendr\'uscolo]{ Departamento de Matem\'atica,
Universidade Federal de S\~ao Carlos, Rodovia Washington Luiz, Km
235, S\~ao Carlos 13565-905, Brazil} 
\thanks{Daniel Vendr\'uscolo was partially supported by grant 2016/24707-4 São Paulo Research Foundation (FAPESP)  }
\email{daniel@dm.ufscar.br}

\date{May 2022}


\begin{abstract}
    We say a group $G$ has property $R_\infty$ if the number $R(\varphi)$ of twisted conjugacy classes is infinite for every automorphism $\varphi$ of $G$. For such groups, the $R_\infty$-nilpotency degree is the least integer $c$ such that $G/\gamma_{c+1}(G)$ has property $R_\infty$. In this work, we compute the $R_\infty$-nilpotency degree of all Generalized Solvable Baumslag-Solitar groups $\Gamma_n$. Moreover, we compute the lower central series of $\Gamma_n$, write the nilpotent quotients $\Gamma_{n,c}=\Gamma_n/\gamma_{c+1}(\Gamma_n)$ as semidirect products of finitely generated abelian groups and classify which integer invertible matrices can be extended to automorphisms of $\Gamma_{n,c}$.
\end{abstract}

\maketitle

\section{Introduction}

The main task of our paper is to show the following:

\begin{teo}[Theorem 4.4]
Let an integer $n \geq 1$ have a prime decomposition with at least two prime numbers involved. Then the $R_\infty$-nilpotency degree of any Generalized Solvable Baumslag-Solitar group $\Gamma_n$ is infinite.
\end{teo}

Let us first define property $R_\infty$ and the $R_\infty$-nilpotency degree and then problematize. An automorphism $\varphi$ of a group $G$ gives rise to a ``twisted conjugacy'' equivalence relation on $G$ given by $x \sim_\varphi y \ \iff\ \exists\ z \in G:\ zx\varphi(z)^{-1}=y$. The number of equivalence classes (or Reidemeister classes, or twisted conjugacy classes) is denoted by $R(\varphi)$, and we say $G$ has property $R_\infty$ if $R(\varphi)$ is infinite for every automorphism $\varphi$ of $G$. This twisted conjugacy relation first appeared in a work of K. Reidemeister \cite{Reidemeister} and has many connections with other areas of mathematics, in particular to fixed point theory (see \cite{Jiang}). We refer the introduction of the paper \cite{DaciDekimpe2} to a good exhibition of those connections and a discussion of the historical context and development of $R_\infty$, including a list of families of groups with this property. The search of $R_\infty$-groups is still very active. For example, a first proof of $R_\infty$ for the pure Artin braid groups $P_n$, $n \geq 3$, was published in 2021 \cite{DGO} and, even more recently, an alternative proof was obtained in \cite{MS}.

For groups $G$ with property $R_\infty$, the $R_\infty$-nilpotency degree is the least integer $c$ such that the quotient $G/\gamma_{c+1}(G)$ has property $R_\infty$, where $\gamma_k(G)$ are the terms of the lower central series of $G$, that is, $\gamma_1(G)=G$ and $\gamma_{c+1}(G)=[\gamma_c(G),G]$, for any $c\geq 1$. If none of the quotients have $R_\infty$, we say the $R_\infty$-nilpotency degree of $G$ is infinite. This degree relates to the folowing problem: it is well known that if $G$ has a characteristic subgroup $N$ (for example, $N=\gamma_k(G)$, $k \geq 1$) such that $G/N$ has $R_\infty$, then $G$ has $R_\infty$. One can ask then whether the converse is true, that is, are there groups $G$ whose property $R_\infty$ is still valid for the quotients $G/\gamma_k(G)$? The answer to this question is positive and gives us the more specific motivation for our work, as we will discuss now. 

Knowing that the free groups $F_n$ have $R_\infty$ (since they are hyperbolic \cite{LevittLustig}), the authors K. Dekimpe and D. L. Gon\c calves studied property $R_\infty$ for infinitely generated free groups, free nilpotent groups and free solvable groups \cite{DG}. They generalized a result of \cite{Rom}, showing that the free nilpotent groups $F_{n,c}=F_n/\gamma_{c+1}(F_n)$, that is, nilpotent quotients of free groups $F_n$, have property $R_\infty$ if and only if $c \geq 2n$. So, the $R_\infty$-nilpotency degree of free groups $F_n$ is $2n$. This fact is actually what motivated the definition of $R_\infty$-nilpotency degree above, which was first given in \cite{DaciDekimpe2}. In that work, also motivated by the fact that fundamental groups of hyperbolic surfaces have $R_\infty$ (again by \cite{LevittLustig}), the same authors studied nilpotent quotients of such surface groups, showing that the $R_\infty$-nilpotency degree $c$ of an orientable surface $S_g$ of genus $g\geq 2$ is $c=4$ and that, for $N_g$ a connected sum of $g\geq 3$ projective planes, $c=2(g-1)$. It is worth noticing that a consequence of their research was the discovery of new examples of nilmanifolds on which every self-homotopy equivalence can be deformed into a fixed point free map.

On the non-hyperbolic side, an important example of $R_\infty$ groups are the Baumslag-Solitar groups
\[
BS(m,n)=\left<a,t\ |\ ta^mt^{-1}=a^n \right>,
\]
for $m,n$ integers (see \cite{DaciFel}). These are important examples of combinatorial and geometric group theory. In \cite{DaciDekimpe}, K. Dekimpe and D. L. Gon\c calves determined the $R_\infty$-nilpotency degree of $BS(m,n)$:

\begin{teo}[Theorem 5.4 in \cite{DaciDekimpe}] Let $0 < m \leq |n|$ with $m \neq n$ and take $d=gcd(m,n)$. Let $p$ denote the largest integer such that $2^p|2\frac{m}{d} + 2$. Then, the $R_\infty$-nilpotency degree $r$ of $BS(m,n)$ is given by

\begin{itemize}
    \item In case $n < 0$ and $n \neq -m$, then $r = 2$.
\item In case $n = -m$ then $r = \infty$.
\item In case $n = m$ then $r = \infty$.
\item In case $n - m = d$, then $r = \infty$.
\item In case $n - m = 2d$, then $2 \leq r \leq p + 2$.
\item In case $n - m \geq 3d$, then $r = 2$.
\end{itemize}
\end{teo}

In the particular case of the solvable Baumslag-Solitar groups $BS(1,n)$, $n \geq 2$, we have $r=\infty$ for $n=2$, $r=4$ for $n=3$ and $r=2$ for $n \geq 4$. The investigation of the $R_\infty$-nilpotency degree for generalizations of these groups is therefore natural. The authors comment about the family of Generalized Baumslag-Solitar groups (or GBS groups), which are $R_\infty$ groups. In this paper, we add one more important family to that discussion, namely, the Generalized Solvable Baumslag-Solitar groups $\Gamma_n$. Let $n\geq 2$ be an integer with prime decomposition $n={p_1}^{y_1}...{p_r}^{y_r}$, the $p_i$ being pairwise distinct and $y_i>0$. We consider the Generalized Solvable Baumslag-Solitar group
\[
\Gamma_n=\left<a,t_1,...,t_r\ |\ t_it_j=t_jt_i,\ i\neq j,\ t_ia{t_i}^{-1}=a^{{p_i}^{y_i}},\ i=1,...,r \right>.
\]

In particular, if $n=p^y$ involves only one prime number, then $\Gamma_n=BS(1,n)$ is a solvable Baumslag-Solitar group. We will focus, therefore, in the case of two or more prime numbers in the decomposition of $n$.

These groups arise mainly in a geometric generalization of the solvable Baumslag-Solitar groups $BS(1,n)$. Consider the group $PSL_2(\mathbb{Z}[1/n])$, which
acts on the product of the hyperbolic space $H_2$ with the product of the Bruhat-Tits trees for $PSL_2(\mathbb{Q}_{p_i})$. The stabilizer of a point at infinity under this action is the upper
triangular subgroup of $PSL_2(\mathbb{Z}[1/n])$, which we are denoting by $\Gamma_n$ (see \cite{TabackWong}). In this paper, however, dealing with the above presentation was enough to estabilish our results. The $\Gamma_n$ are also known to be solvable and metabelian groups, fitting in a short exact sequence of the form

\[
  1  \to \mathbb{Z}\left[\frac{1}{n}\right] \to \Gamma_n \to \mathbb{Z}^r \to 1.
\] More details about $\Gamma_n$ can be seen in the paper \cite{TabackWong}, where J. Taback and P. Wong show property for $\Gamma_n$ and for any group quasi-isometric to it.

We started dealing with Baumslag-Solitar groups in the first author's doctoral dissertation. There, one realized that elementary techniques involving the BNS invariant $\Sigma^1$ of the groups $BS(1,n)$ could also be applied to $\Gamma_n$, leading to an elementary proof of $R_\infty$ to these groups. One can then ask whether other techniques - such as the investigation of nilpotent quotients of $BS(m,n)$ in \cite{DaciDekimpe} - can also be adapted to $\Gamma_n$. The main question was whether $\Gamma_n$ would behave more like $BS(1,2)$ (with infinite $R_\infty$-nilpotency degree) or like $BS(1,n)$, $n\geq 3$ (finite degree cases). As Theorem 1 shows, the groups $\Gamma_n$ fit in the infinite degree case. We believe this happens because of the large automorphism groups of its nilpotent quotients.

We will denote the nilpotent quotients of $\Gamma_n$ by
\[
\Gamma_{n,c}=\frac{\Gamma_n}{\gamma_{c+1}(\Gamma_n)}
\] for any $c \geq 1$, where $\gamma_{c+1}$ is the $(c+1)^{th}$ term of the lower central series of $\Gamma_n$. We know that $\Gamma_{n,c}$ is a nilpotent group with nilpotency class $\leq c$, for $\gamma_{c+1}(\Gamma_{n,c})=\gamma_{c+1}(\frac{\Gamma_n}{\gamma_{c+1}(\Gamma_n)})=\frac{\gamma_{c+1}(\Gamma_n)}{\gamma_{c+1}(\Gamma_n)}=1$. In this work, the torsion subgroup of a nilpotent group $G$ will be denoted by $\tau G=\{g \in G\ |\ g^n=1\ \text{for some}\ n \geq 1\}$.

This work is divided as follows: in section 2, we compute the torsion subgroup and the lower central series of $\Gamma_{n,c}$. Then, in section 3, we use these computations to create an useful isomorphism between $\Gamma_{n,c}$ and a semidirect product of the form $G_{n,c}= \Z_{m^c} \rtimes \Z^r$, for some $m\geq 1$. Finally, in section 4, we classify which matrices in $GL_r(\mathbb{Z})$ can be extended to automorphisms of $\Gamma_{n,c}$ and use the classification to find specific automorphisms with finite Reidemeister numbers in $\Gamma_{n,c}$.

\section{Torsion and lower central series of $\Gamma_{n,c}$}

We begin by showing the following

\begin{lema}\label{nq1}
Let $m=\gcd(p_1^{y_1}-1,...,p_r^{y_r}-1)$. Then $a^{m^k} \in \gamma_{k+1}(\Gamma_n)$ for all $k \geq 1$.
\end{lema}

\begin{proof}
Induction on $k$. First, $k=1$. By using the group relations, note that, for any $1 \leq i \leq r$, $a^{p_i^{y_i}-1}=t_iat_i^{-1}a^{-1}=[t_i,a] \in \gamma_2(\Gamma_n)$. Since this is true for any $i$ and $m$ is an integer combination of the $p_i^{y_i}-1$, we have $a^m \in \gamma_2(\Gamma_n)$. Now, suppose the lemma is true for some $k \geq 1$. Then
\[
a^{(p_i^{y_i}-1)m^k}=a^{p_i^{y_i}m^k}a^{-m^k}=t_ia^{m^k}t_i^{-1}a^{-m^k}=[t_i,a^{m^k}] \in \gamma_{k+2}(\Gamma_n).
\] Again, since this is true for any $i$ and $m$ is an integer combination of the $p_i^{y_i}-1$, we have $a^{mm^k}=a^{m^{k+1}} \in \gamma_{k+2}(\Gamma_n)$, as desired. This completes the proof.
\end{proof}

To compute the torsion of the groups $\Gamma_{n,c}$, we need the following standard results:

\begin{prop}[see \cite{Mikhailov}] \label{contascomcomut}
Let $k,m,n\geq 1$ and let $x,y,z\in G$ be elements of a group $G$ such that $x\in \gamma_k(G)$, $y\in \gamma_m(G)$ and $z\in \gamma_n(G)$. Then: 
\begin{itemize}
\item[a)] $xy=yx$ $\mod \gamma_{k+m}(G)$;
\item[b)] $[x,yz]=[x,y][x,z]$ $\mod \gamma_{k+m+n}(G)$;
\item[c)] $[xy,z]=[x,z][y,z]$ $\mod \gamma_{k+m+n}(G)$.
\end{itemize}
\end{prop}

\begin{prop}(\cite{MagKarSol}, Theorem 5.4) \label{generatorsofquotient}
If $G$ is finitely generated by elements $x_1,\dots ,x_r$ then, for any $k\geq 1$, $\gamma_k(G)/\gamma_{k+1}(G)$ is abelian and finitely generated by the cosets of the $k$-fold commutators $[x_{i_1},\dots ,x_{i_k}]$, where $1\leq i_j\leq r$.
\end{prop}

By an easy recursive argument, one shows the following:

\begin{lema}\label{nq2}
Let $G$ be a nilpotent group of class $\leq c$ and denote $\gamma_i=\gamma_i(G)$. If the quotients $\gamma_2/\gamma_3,...,\gamma_c/\gamma_{c+1}$ are finite, then $\gamma_2$ is a torsion subgroup of $G$. \qed
\end{lema}


\begin{prop}
$\tau \Gamma_{n,c}= \left< \overline{a},\gamma_2(\Gamma_{n,c}) \right>$, where $\overline{a}=a\gamma_{c+1}=a\gamma_{c+1}(\Gamma_n) \in \Gamma_{n,c}$.
\end{prop}

\begin{proof}
In the case $c=1$ we have $\Gamma_{n,1}$ is the abelianized group of $\Gamma_n$, so
\[
\Gamma_{n,1}=\left<\overline{a},\overline{t_1},...,\overline{t_r}\ |\ \overline{t_i}\overline{t_j}=\overline{t_j}\overline{t_i},\ \overline{t_i}\overline{a}=\overline{a}\overline{t_i},\ \overline{a}^{p_i^{y_i}-1}=1 \right> \simeq \Z_m \times \Z^r,
\] 
where $m=\gcd(p_1^{y_1}-1,...,p_r^{y_r}-1)$. So $\tau \Gamma_{n,1}=\left< \overline{a} \right>=\left< \overline{a},\gamma_2(\Gamma_{n,1}) \right>$, since $\gamma_2(\Gamma_{n,1})=1$.

Now let us show the proposition in the case $c \geq 2$. For $(\subset)$, let $x\gamma_{c+1} \in \tau \Gamma_{n,c}$. This means $x^k\in \gamma_{c+1}$ for some $k\geq 1$. Since $c \geq 2$, we have $x^k \in \gamma_{c+1} \subset \gamma_2$, so $x\gamma_2 \in \tau \Gamma_{n,1}=\left< \overline{a} \right>$. Write then $x=a^lg_2$ for $l \in \Z$ and $g_2 \in \gamma_2=\gamma_2(\Gamma_n)$. This gives $x\gamma_{c+1}=(a\gamma_{c+1})^l(g_2\gamma_{c+1}) \in \left< \overline{a}, \gamma_2(\Gamma_{n,c}) \right>$, as we wanted. To show $(\supset)$, we note that by Lemma \ref{nq1} we get $\overline{a}^{m^c}=1$ in $\Gamma_{n,c}$, so $\overline{a} \in \tau \Gamma_{n,c}$. So, we just need to show that $\gamma_2(\Gamma_{n,c})$ is a torsion subgroup of $\Gamma_{n,c}$. To do this, we invoke Lemma \ref{nq2}, by which we know it is enough to show the quotients
\[
\frac{\gamma_2(\Gamma_{n,c})}{\gamma_3(\Gamma_{n,c})},...,\frac{\gamma_c(\Gamma_{n,c})}{\gamma_{c+1}(\Gamma_{n,c})}
\] are all finite. But for every $2 \leq i \leq c$, by the known Isomorphism Theorem for quotients, we have
\[
\frac{\gamma_i(\Gamma_{n,c})}{\gamma_{i+1}(\Gamma_{n,c})}=\frac{\gamma_i(\Gamma_n)/\gamma_{c+1}(\Gamma_n)}{\gamma_{i+i}(\Gamma_n)/\gamma_{c+1}(\Gamma_n)} \simeq \frac{\gamma_i(\Gamma_n)}{\gamma_{i+i}(\Gamma_n)}=\gamma_i/\gamma_{i+1},
\] so let us show that $\gamma_2/\gamma_3,...,\gamma_c/\gamma_{c+1}$ are finite by induction. By Proposition \ref{generatorsofquotient}, we know they are abelian groups, generated by their $i$-fold comutator cosets. The group $\gamma_2/\gamma_3$ is generated by the elements $[t_i,a]\gamma_3$, $1 \leq i \leq r$ and by $[t_i,t_j]\gamma_3=1\gamma_3=\gamma_3$, which are trivial. By Proposition \ref{contascomcomut}, we get $[t_i,a]^m \gamma_3=[t_i,a^m]\gamma_3=\gamma_3$, since $a^m \in \gamma_2$ (Lemma \ref{nq1}). So all generators of $\gamma_2/\gamma_3$ have torsion. Since it is finitely generated and abelian, it must be a finite group. Finally, suppose by induction that $\gamma_i/\gamma_{i+1}$ is finite for some $i \geq 2$. By Proposition \ref{generatorsofquotient}, $\gamma_{i+1}/\gamma_{i+2}$ is then generated by the elements of the form $[x,y]\gamma_{i+2}$ with $x \in \gamma_i$ and $y \in \Gamma_n$. Since $\gamma_i/\gamma_{i+1}$ is finite, let $k=k(x,y) \geq 1$ such that $x^k \in \gamma_{i+1}$. Then $[x,y]^k \gamma_{i+2}=[x^k,y]\gamma_{i+2}=\gamma_{i+2}$. By the same argument we just used, this implies $\gamma_{i+1}/\gamma_{i+2}$ is finite and completes the proof.
\end{proof}

\begin{prop}\label{nqlcs}
$\gamma_k(\Gamma_{n,c})=\left< \overline{a}^{m^{k-1}} \right>$ for all $k \geq 2$ and $c \geq 1$.
\end{prop}

\begin{proof}
First, we will show that
\begin{equation}\label{nqeq1}
\frac{\gamma_k(\Gamma_{n,c})}{\gamma_{k+1}(\Gamma_{n,c})}=\left< \overline{a}^{m^{k-1}}\gamma_{k+1}(\Gamma_{n,c}) \right>.
\end{equation} For $k=2$, by Proposition \ref{generatorsofquotient}, $\frac{\gamma_2(\Gamma_{n,c})}{\gamma_3(\Gamma_{n,c})}$ is generated by the cosets $[\overline{t_i},\overline{a}]\gamma_3(\Gamma_{n,c})$. Since $[\overline{t_i},\overline{a}]=\overline{a}^{p_i^{y_i}-1}$, we have
\[
\frac{\gamma_2(\Gamma_{n,c})}{\gamma_3(\Gamma_{n,c})}=\left< \overline{a}^{p_1^{y_1}-1}\gamma_3(\Gamma_{n,c}),...,\overline{a}^{p_r^{y_r}-1}\gamma_3(\Gamma_{n,c}) \right>=\left< \overline{a}^m \gamma_3(\Gamma_{n,c}) \right>
\](remember that $m=\gcd(p_1^{y_1}-1,...,p_r^{y_r}-1)$). Suppose now (\ref{nqeq1}) is true for some $k \geq 2$. We know $\frac{\gamma_{k+1}(\Gamma_{n,c})}{\gamma_{k+2}(\Gamma_{n,c})}$ is generated by the cosets $[x,z]\gamma_{k+2}(\Gamma_{n,c})$, where $x \in \gamma_k(\Gamma_{n,c})$ and $z \in \Gamma_{n,c}$. By induction, we can write $x=\overline{a}^{\alpha m^{k-1}}w_{k+1}$ for some $w_{k+1} \in \gamma_{k+1}(\Gamma_{n,c})$ and $\alpha \in \Z$. Then, by using Proposition \ref{contascomcomut} we get
\begin{eqnarray*}
[x,z]\gamma_{k+2}(\Gamma_{n,c}) &=& [\overline{a}^{\alpha m^{k-1}}w_{k+1},z]\gamma_{k+2}(\Gamma_{n,c})\\
								&=& [\overline{a}^{m^{k-1}},z]^\alpha[w_{k+1},z]\gamma_{k+2}(\Gamma_{n,c})\\
								&=&[\overline{a}^{m^{k-1}},z]^\alpha \gamma_{k+2}(\Gamma_{n,c}),
\end{eqnarray*} so the quotient $\frac{\gamma_{k+1}(\Gamma_{n,c})}{\gamma_{k+2}(\Gamma_{n,c})}$ is actually generated only by the cosets $[\overline{a}^{m^{k-1}},z] \gamma_{k+2}(\Gamma_{n,c})$. Since $[\overline{a}^{m^{k-1}},\overline{a}]$ is obviously trivial, the quotient group is generated only by the generators $[\overline{a}^{m^{k-1}},\overline{t_i}]\gamma_{k+2}(\Gamma_{n,c})$. Since $[\overline{a}^{m^{k-1}},\overline{t_i}]=\overline{a}^{(p_i^{y_i}-1)m^{k-1}}$, we obtain
\[
\frac{\gamma_{k+1}(\Gamma_{n,c})}{\gamma_{k+2}(\Gamma_{n,c})}=\left< \overline{a}^{(p_1^{y_1}-1)m^{k-1}}\gamma_{k+2}(\Gamma_{n,c}),...,\overline{a}^{(p_r^{y_r}-1)m^{k-1}}\gamma_{k+2}(\Gamma_{n,c}) \right>=\left< \overline{a}^\beta \gamma_{k+2}(\Gamma_{n,c}) \right>,
\] where
\[
\beta=\gcd((p_1^{y_1}-1)m^{k-1},...,(p_r^{y_r}-1)m^{k-1})=m^{k-1}\gcd(p_1^{y_1}-1,...,p_r^{y_r}-1)=m^k,
\] and this shows (\ref{nqeq1}). Now, let us show the proposition. The $(\supset)$ part is a direct consequence of Lemma \ref{nq1}. Let us show $(\subset)$. In the case $c<k$, we have $\gamma_k(\Gamma_{n,c})=1 \subset \left< \overline{a}^{m^{k-1}} \right>$. Suppose then $c\geq k$ and let $x \in \gamma_k(\Gamma_{n,c})$. Since $x\gamma_{k+1}(\Gamma_{n,c}) \in \left< \overline{a}^{m^{k-1}}\gamma_{k+1}(\Gamma_{n,c}) \right>$ (by (\ref{nqeq1})), write $x=\overline{a}^{j_km^{k-1}}x_{k+1}$ for $j_k \in \Z$ and $x_{k+1} \in \gamma_{k+1}(\Gamma_{n,c})$. By using (\ref{nqeq1}) again, we write $x_{k+1}=\overline{a}^{j_{k+1}m^k}x_{k+2}$ for $j_{k+1} \in \Z$ and $x_{k+2} \in \gamma_{k+2}(\Gamma_{n,c})$. We can do this recursively to obtain
\begin{eqnarray*}
x &=& \overline{a}^{j_km^{k-1}}\overline{a}^{j_{k+1}m^k}...\overline{a}^{j_cm^{c-1}}x_{c+1}\\
  &=& \overline{a}^{m^{k-1}(j_k+j_{k+1}m+...+j_c m^{c-k})}\\
  & \in & \left< \overline{a}^{m^{k-1}} \right>,
\end{eqnarray*}and the proof is complete.
\end{proof}

By Lemma \ref{nq1} and the two propositions above, we get

\begin{cor}\label{nq3}
$\tau \Gamma_{n,c} = \left< \overline{a} \right>$ and $\card(\tau \Gamma_{n,c}) \leq m^c$. \qed
\end{cor}

\section{An isomorphism for $\Gamma_{n,c}$}

The next step is to find a presentation to $\Gamma_{n,c}$, so we will find an isomorphism between $\Gamma_{n,c}$ and a more known group. We will use the notations from the previous section and will also denote $\Z_{m^c}=\left< x\ |x^{m^c}=1 \right>$ and $\Z^r=\left< s_1,...,s_r\ | s_is_j=s_js_i\right>$. We define the group
\[
G_{n,c}= \Z_{m^c} \rtimes \Z^r,
\] where the action of $\Z^r$ on $\Z_{m^c}$ is given by $s_i x s_i^{-1}=x^{p_i^{y_i}}, 1 \leq i \leq r$.

\begin{obs}
Note first that the actions defined above are all automorphisms of $\Z_{m^c}$, since $\gcd(p_i^{y_i},m)=1$ (and so $\gcd(p_i^{y_i},m^c)=1$ for any $c\geq 1$). Second, all such automorphisms commute, for $\Z_{m^c}$ is cyclic. These facts show that there is a well defined homomorphism $\Z^r \to Aut(\Z_{m^c})$, so this semidirect product is well defined.
\end{obs}

We will show that $\Gamma_{n,c} \simeq G_{n,c}$. To do this, we need:

\begin{lema}
$G_{n,c}$ is nilpotent of class $\leq c$.
\end{lema}

\begin{proof}
Since $[s_i,x]=x^{p_i^{y_i}-1} \in \left< x^m \right>$ for every $i$, we have $\gamma_2(G_{n,c})\subset \left< x^m \right>$. Similarly, since $[s_i,x^m]=x^{(p_i^{y_i}-1)m} \in \left< x^{m^2} \right>$ for every $i$, in particular we have $[s_i,z] \in \left< x^{m^2}\right>$ for every $z \in \gamma_2(G_{n,c})$, so it is easy to see that $\gamma_3(G_{n,c})\subset \left< x^{m^2} \right>$. Recursively, we can show that $\gamma_k(G_{n,c})\subset \left< x^{m^{k-1}}\right>$ for every $k \geq 2$. In particular, $\gamma_{c+1}(G_{n,c})\subset \left< x^{m^c}\right>=1$, since $x^{m^c}=1$ in $\Z_{m^c}$. This shows the lemma.
\end{proof}

\begin{cor}\label{nq4}
$\tau G_{n,c}$ is a subgroup of $G_{n,c}$. Moreover, $\tau G_{n,c}=\Z_{m^c}= \left< x \right>$ and so $\card(\tau G_{n,c})=m^c$.
\end{cor}

\begin{teorema}\label{nq4.5}
$\Gamma_{n,c} \simeq G_{n,c}$.
\end{teorema}

\begin{proof}
Let $f:\Gamma_n \to G_{n,c}$ be the map $f(a)=x$ and $f(t_i)=s_i$. Since $f(t_i)f(a)f(t_i)^{-1}=s_ixs_i^{-1}=x^{p_i^{y_i}}=f(a)^{p_i^{y_i}}$, $f$ is a well defined group homomorphism. Since $f(\gamma_i(\Gamma_n)) \subset \gamma_i(G_{n,c})$, $f$ induces the morphism
\[
f:\Gamma_{n,c}=\frac{\Gamma_n}{\gamma_{c+1}(\Gamma_n)} \to \frac{G_{n,c}}{\gamma_{c+1}(G_{n,c})}=G_{n,c}
\] given by $f(\overline{a})=x$ and $f(\overline{t_i})=s_i$. It is obviously surjective. We are just left to show that $\ker(f)=1$, and to do that we will make use of the torsion subgroups. Since $f(\tau \Gamma_{n,c})\subset \tau G_{n,c}$ (this is true for any homomorphisms between nilpotent groups), there is the restriction morphism $f_\tau: \tau \Gamma_{n,c} \to \tau G_{n,c}$. By Corollaries \ref{nq3} and \ref{nq4}, we can actually write $f_\tau: \left< \overline{a} \right> \to \left< x \right>$. Since $f_\tau(\overline{a})=x$, it is clearly surjective. Now, $f_\tau$ is a surjective map from a finite set of $\leq m^c$ elements (Corollary \ref{nq3}) to a finite set with exactly $m^c$ elements (Corollary \ref{nq4}), so we must have $\card(\left< \overline{a} \right>)=m^c$ and $f_\tau$ an isomorphism. In particular, $\ker(f_\tau)=1$. We claim that $\ker(f) \subset \tau \Gamma_{n,c}$. In fact, let $z \in \ker(f)$. By using the relations in $\Gamma_n$, we can write

\[
z=\overline{t_1}^{k_1}...\overline{t_r}^{k_r}\overline{t_1}^{-\alpha_1}...\overline{t_r}^{-\alpha_r}\overline{a}^l {t_r}^{\alpha_r}...\overline{t_1}^{\alpha_1},
\] for $k_i,l \in \Z$ and $\alpha_i \geq 0$. So
\[
1=f(z)=s_1^{k_1}...s_r^{k_r}s_1^{-\alpha_1}...s_r^{-\alpha_r}x^ls_r^{\alpha_r}...s_1^{\alpha_1}.
\] Since $x \in \tau G_{n,c} \lhd G_{n,c}$ we have $s_1^{-\alpha_1}...s_r^{-\alpha_r}x^ls_r^{\alpha_r}...s_1^{\alpha_1} \in \tau G_{n,c}=\left< x \right>$, so $1=f(z)=s_1^{k_1}...s_r^{k_r}x^{l'}$ for some $l' \in \Z$. By projecting this equality under the natural homomorphism $G_{n,c} \to \Z^r$ we get $1=s_1^{k_1}...s_r^{k_r}$, which implies $k_i=0$ for every $i$. Therefore $z=\overline{t_1}^{-\alpha_1}...\overline{t_r}^{-\alpha_r}\overline{a}^l \overline{t_r}^{\alpha_r}...\overline{t_1}^{\alpha_1} \in \tau G_{n,c}$, since $\overline{a} \in \tau \Gamma_{n,c} \lhd \Gamma_{n,c}$, which shows the claim. Finally, this gives $\ker(f)=\ker(f)\cap \tau \Gamma_{n,c}=\ker(f_\tau)=1$ and the theorem is proved.
\end{proof}

\begin{cor}\label{nqcorpres}
For any $c \geq 1$, the nilpotent quotient $\Gamma_{n,c}$ has the following presentation:
\[
\Gamma_{n,c}=\left< x,s_1,...,s_r\ |\ x^{m^c}=1,\ s_is_j=s_js_i,\ s_i x s_i^{-1}=x^{p_i^{y_i}}  \right>.
\]\qed
\end{cor}


\section{Reidemeister numbers}

Because of the theorem above, from now on we will make the following identifications
\[
\Gamma_{n,c}=G_{n,c}=\Z_{m^c} \rtimes \Z^r=\left<x\right> \rtimes \left<s_1,...,s_r \right>.
\] It's also worth remembering that we will restrict us to investigate Reidemeister numbers of $\Gamma_{n,c}$ only in the case $r \geq 2$, for, if $r=1$, then $\Gamma_n$ is by definition a Baumslag-Solitar group $BS(1,n)$ and its Reidemeister numbers were studied in \cite{DaciDekimpe}. Let $\varphi \in Aut(\Gamma_{n,c})$. Since $\varphi(\tau \Gamma_{n,c})\subset \tau \Gamma_{n,c}$, we have an induced automorphism
\[
\overline{\varphi}:\frac{\Gamma_{n,c}}{\tau \Gamma_{n,c}}=\Z^r \to \Z^r=\frac{\Gamma_{n,c}}{\tau \Gamma_{n,c}}.
\] From now on, we will use the usual identification $Aut(\Z^r)=GL_r(\Z)$ which sees an automorphism of $\Z^r$ as its (integer invertible) matrix with respect to the coordinates $s_i$. So, if $\overline{\varphi}(s_i)=s_1^{a_{1i}}...s_r^{a_{ri}}$, we will identify
\[
\overline{\varphi}=(a_{ij})_{ij}=\begin{bmatrix} 
					a_{11} & \cdots 	& a_{1r} \\
					\vdots 		&  			& \vdots \\
					a_{r1} & \cdots 	& a_{rr} \\
\end{bmatrix}=\left[ A_1 \cdots A_r \right],\ \ \text{where}\ A_i=\begin{bmatrix} 
					a_{1i} \\
					\vdots		\\
					a_{ri} \\
\end{bmatrix} \in \Z^r.
\]

\begin{prop}\label{nq5}
If $\varphi \in Aut(\Gamma_{n,c})$, the following are equivalent:
\begin{itemize}
\item[(1)]$R(\varphi)=\infty$;
\item[(2)]$R(\overline{\varphi})=\infty$;
\item[(3)]$\det(\overline{\varphi}-Id) = 0$;
\item[(4)]$\overline{\varphi}$ has $1$ as an eigenvalue.
\end{itemize}
\end{prop}

\begin{proof}

Items $(2),(3)$ and $(4)$ are well known to be all equivalent. Also, we have an obvious commutative diagram involving the automorphisms $\varphi$, $\overline{\varphi}$ and the projection $\pi:\Gamma_{n,c}\to \mathbb{Z}^r$. Thus, by Lemma 1.1 of \cite{DaciWong}, it follows that $(2)$ implies $(1)$. So we only have to prove that $(1)$ implies $(2)$.

To simplify the computation, let us use the following notation in this proof: given $y=(y_1,...,y_r) \in \Z^r$ (either a row or a column vector), we will denote the element $s_1^{y_1}...s_r^{y_r} \in \Gamma_{n,c}$ by $S^y$, and the scalar product of $k \in \Z$ by $y$ is denoted by $ky$. With this notation, it turns out that any element of $\Gamma_{n,c}$ is of the form $S^y x^\beta$ for some $y \in \Z^r$ and $\beta \in \Z$. Suppose then that $R(\overline{\varphi})=d<\infty$ and write $\mathcal{R}(\overline{\varphi})=\{[v_1]_{\overline{\varphi}},...,[v_d]_{\overline{\varphi}}\}$ for $v_i \in \Z^r$ or, equivalently, $\frac{\Z^r}{\im(\overline{\varphi}-Id)}=\{\overline{v_1},...,\overline{v_d}\}$ (where $\overline{v_i}=v_i+\im(\overline{\varphi}-Id)$). Write $\varphi(x)=x^\mu$ (for some $\mu \in \Z$ with $\gcd(\mu,m^c)=1$) and $\varphi(s_i)=S^{A_i}x^{\beta_i}$, $\beta_i \in \Z$. Given that the $s_i$-coordinates behave well in the $\Gamma_{n,c}$, for any $k=(k_1,...,k_r) \in \Z^r$ and $l \in \Z$ we have
\[ \varphi(S^kx^l)=S^{\overline{\varphi}(k)}x^{\theta},\]
for some $\theta \in \Z$. This implies that, for any $j \in \Z$ and $y \in \Z^r$,
\[ (S^kx^l)(S^yx^j)\varphi(S^kx^l)^{-1}=  S^{y+(Id-\overline{\varphi})(k)}x^{\tilde{\theta}},\]
for some $\tilde{\theta}\in \Z$. This means that, if two vectors $y,y' \in \Z^r$ are such that $\overline{y}=\overline{y'} \in \frac{\Z^r}{\im(\overline{\varphi}-Id)}$, then every element $S^yx^j$ is $\varphi$-conjugated to some element $S^{y'}x^\theta$ for some $0 \leq \theta < m^c$. Since $\frac{\Z^r}{\im(\overline{\varphi}-Id)}=\{\overline{v_1},...,\overline{v_d}\}$, every element $S^yx^j$ is $\varphi$-conjugated to some $S^{v_i}x^{\theta}$, $1 \leq i \leq d$, $0 \leq \theta < m^c$, so $R(\varphi) \leq dm^c <\infty$ and the proposition is proved.

\end{proof}

In the rest of the work we will use the following notation: we know that $\gcd(p_i^{y_i},m^c)=1$. This means that $p_i^{y_i}$ is an invertible element in the commutative ring $\Z_{m^c}$ (now thought in the abelian notation $\Z_{m^c}=\{0,1,...,m^c-1\}$). So, we will naturally denote by $p_i^{-y_i}$ the inverse element $(p_i^{y_i})^{-1} \in \Z_{m^c}$ and, similarly, we define $p_i^{-ky_i}$ as $(p_i^{ky_i})^{-1}$ for any $k \geq 0$, so it makes sense to write $p_i^{ky_i}$ for any $k \in \Z$, thinking of it as an invertible element of the ring $\Z_{m^c}$. We are saying this to avoid a possible misinterpretation of $p_i^{-y_i}$ as $\frac{1}{p_i^{y_i}} \in \Q$, for example. With this notation, it is clear that $s_i^k x s_i^{-k}=x^{p_i^{ky_i}}$ for any $k \in \Z$.

\begin{prop}\label{nq6}
$\Gamma_{n,c}$ has not property $R_\infty$ if and only if there is $M=(a_{ij})_{ij} \in Gl_r(\Z)$ such that
\begin{itemize}
\item $\det(M-Id) \neq 0$;
\item for any $1 \leq i \leq r$,
\[
p_1^{a_{1i}y_1}p_2^{a_{2i}y_2}\dots p_r^{a_{ri}y_r}\equiv p_i^{y_i}\ \ \text{mod}\ m^c.\ \ \ \ \ (M,c,i)
\]
\end{itemize}
\end{prop}

\begin{proof} Suppose first that $\Gamma_{n,c}$ has not property $R_\infty$. Let $\varphi \in Aut(\Gamma_{n,c})$ such that $R(\varphi)<\infty$. Let $M=\overline{\varphi} \in Gl_r(\Z)$, and write $M=(a_{ij})_{ij}$. By Proposition \ref{nq5}, we have $\det(M-Id)\neq 0$. Since $\varphi(\tau \Gamma_{n,c}) \subset \tau \Gamma_{n,c}$, we have $\varphi(x)=x^\mu$ for some $\mu \in \Z$ such that $\gcd(\mu,m^c)=1$. Let us show that for any $1 \leq i \leq r$ the equation $(M,c,i)$ holds. For any such $i$, since $\varphi$ is a homomorphism of $\Gamma_{n,c}$ it must satisfy $\varphi(s_i)\varphi(x)\varphi(s_i)^{-1}=\varphi(x)^{p_i^{y_i}}$, so
\[
s_1^{a_{1i}}\dots s_r^{a_{ri}}x^\mu s_r^{-a_{ri}}\dots s_1^{-a_{1i}}=x^{\mu p_i^{y_i}}
\]or, equivalently,
\[
x^{\mu p_1^{a_{1i}y_1}\dots p_r^{a_{ri}y_r}}=x^{\mu p_i^{y_i}}.
\] Then $\mu p_1^{a_{1i}y_1}\dots p_r^{a_{ri}y_r}\equiv \mu p_i^{y_i}$ mod $m^c$, and since $\gcd(\mu,m^c)=1$, we have $p_1^{a_{1i}y_1}\dots p_r^{a_{ri}y_r}\equiv p_i^{y_i}$ mod $m^c$, which is exactly $(M,c,i)$. This shows the ``if'' part. Suppose now that there is such a matrix $M=(a_{ij})_{ij}$ and let us show $\Gamma_{n,c}$ has not $R_\infty$. Define $\varphi:\Gamma_{n,c} \to \Gamma_{n,c}$ by $\varphi(x)=x$ and $\varphi(s_i)=s_1^{a_{1i}}s_2^{a_{2i}}...s_r^{a_{ri}}$. Let us check that $\varphi$ is a well defined homomorphism:
\[
\varphi(s_i)\varphi(x)\varphi(s_i)^{-1}=s_1^{a_{1i}}s_2^{a_{2i}}...s_r^{a_{ri}}x s_r^{-a_{ri}}...s_2^{-a_{2i}}s_1^{-a_{1i}}=x^{p_1^{a_{1i}y_1}...p_r^{a_{ri}y_r}}=\varphi(x)^{p_1^{a_{1i}y_1}...p_r^{a_{ri}y_r}}=\varphi(x)^{p_i^{y_i}},
\]the last equality being true by $(M,c,i)$. Also, since the $s_i$ commute, we obviously have
\[
\varphi(s_i)\varphi(s_j)=s_1^{a_{1i}}...s_r^{a_{ri}}s_1^{a_{1j}}...s_r^{a_{rj}}=s_1^{a_{1j}}...s_r^{a_{rj}}s_1^{a_{1i}}...s_r^{a_{ri}}=\varphi(s_j)\varphi(s_i).
\]Finally,
\[
\varphi(x)^{m^c}=x^{m^c}=1,
\] so $\varphi$ is in fact a homomorphism. Let us now construct an inverse homomorphism. Let $N=M^{-1} \in GL_r(\Z)$ and write $N=(b_{ij})_{ij}$. Let us show that, for any $1 \leq i \leq r$, $N$ satisfies the equation $(N,c,i)$, that is $p_1^{b_{1i}y_1}p_2^{b_{2i}y_2}..p_r^{b_{ri}y_r}=p_i^{y_i}$ mod $m^c$. Since $MN=Id$, for any $1 \leq i,j \leq r$ we have
\[
\sum_{k=1}^r a_{ik}b_{kj}=(MN)_{ij}=Id_{ij}=\delta_{ij},
\] where $\delta_{ij}$ is the Kronecker delta. Fix $i$. We do the following: for each fixed $1 \leq j \leq r$, we raise both sides of equation $(M,c,j)$ to the power of $b_{ji}$ and obtain
\[
p_1^{a_{1j}b_{ji}y_1}p_2^{a_{2j}b_{ji}y_2}\dots p_r^{a_{rj}b_{ji}y_r}\equiv p_j^{b_{ji}y_j}\ \ mod\ m^c
\]
Now, if we do the product of all the $r$ equations above (on both sides, of course) and rearrange the left side according to the primes we get
\[
p_1^{(a_{11}b_{1i}+\dots +a_{1r}b_{ri})y_1}p_2^{(a_{21}b_{1i}+\dots +a_{2r}b_{ri})y_2}\dots p_r^{(a_{r1}b_{1i}+\dots +a_{rr}b_{ri})y_r}\equiv p_1^{b_{1i}y_1}p_2^{b_{2i}y_2}\dots p_r^{b_{ri}y_r}\ \ mod\ m^c,
\]or
\[
p_1^{(\sum_{k}a_{1k}b_{ki})y_1}p_2^{(\sum_{k}a_{2k}b_{ki})y_2}\dots p_r^{(\sum_{k}a_{rk}b_{ki})y_r}\equiv p_1^{b_{1i}y_1}p_2^{b_{2i}y_2}\dots p_r^{b_{ri}y_r}\ \ mod\ m^c,
\]or even
\[
p_1^{\delta_{1i}y_1}p_2^{\delta_{2i}y_2}\dots p_r^{\delta_{ri}y_r}\equiv p_1^{b_{1i}y_1}p_2^{b_{2i}y_2}\dots p_r^{b_{ri}y_r}\ \ mod\ m^c,
\]which results in 
\[
p_i^{y_i} \equiv p_1^{b_{1i}y_1}p_2^{b_{2i}y_2}\dots p_r^{b_{ri}y_r}\ \ mod\ m^c,
\]which is exactly $(N,c,i)$, as we wanted. Now define $\psi: \Gamma_{n,c} \to \Gamma_{n,c}$ by $\psi(x)=x$ and $\psi(s_i)=s_1^{b_{1i}}s_2^{b_{2i}}...s_r^{b_{ri}}$. As we did with $\varphi$, the fact that $N$ satisfies $(N,c,i)$ for all $i$ gives us that $\psi$ is a group homomorphism. Of course we have $\varphi(\psi(x))=x$. Also, by the fact that $MN=Id$, straightforward calculations show that $\varphi(\psi(s_i))=s_i$.
Similarly, we show that $\psi \varphi =Id$ by using that $NM=Id$, so $\varphi \in Aut(\Gamma_{n,c})$. Since $\overline{\varphi}=M$ we have $\det(\overline{\varphi}-Id)=\det(M-Id) \neq 0$ by hypothesis, so $R(\varphi)<\infty$ by Proposition \ref{nq5}. This completes the proof.
\end{proof}

\begin{obs}
Implicit in the proof of Proposition \ref{nq6} above is the classification of all matrices in $GL_r(\mathbb{Z})$ which can be extended to automorphisms of $\Gamma_{n,c}$. In other words, given a matrix $M=(a_{ij})_{ij} \in GL_r(\Z)$, there is an automorphism $\varphi$ of $\Gamma_{n,c}$ such that $\overline{\varphi}=M$ if and only if all equations $(M,c,i)$ are satisfied. 
\end{obs}

To proceed, we need the following lemma, which can be easily shown by elementary number theory and induction on $k$:

\begin{lema}
Let $x,m \geq 2$. If $x=1\ \text{mod}\ m$, then $x^{m^k}=1\ mod\  m^{k+1}$ for any $k \geq 0$. \qed
\end{lema}


\begin{teorema}\label{nqteo}
Let $n \geq 2$ have prime decomposition $n={p_1}^{y_1}\dots {p_r}^{y_r}$, the $p_i$ being pairwise distinct and $y_i>0$. Suppose $r \geq 2$, that is, there are at least two primes involved. Then the nilpotent quotient group $\Gamma_{n,c}=\Gamma_n/\gamma_{c+1}(\Gamma_n)$ does not have property $R_\infty$ for any $c \geq 1$. In other words, the $R_\infty$-nilpotency degree of $\Gamma_n$ is infinite.
\end{teorema}

\begin{proof}
Let $m=\gcd(p_1^{y_1}-1,\dots ,p_r^{y_r}-1)$, as we have done in this work. If $m=1$, then none of the groups $\Gamma_{n,c}$ have property $R_\infty$. This is because $\Gamma_{n,c} \simeq \Z^r$ for any $c$ in this case (see Theorem \ref{nq4.5}), and we know $\Z^r$ has not $R_\infty$. So, from now on, suppose $m \geq 2$. Of course $\Gamma_{n,1}$ does not have property $R_\infty$, for it is a finitely generated abelian group. Now, for any fixed $c \geq 2$, we will use Proposition \ref{nq6}, that is, for any $r \geq 2$, we will find a matrix $M=(a_{ij})_{ij} \in Gl_r(\Z)$ with $\det(M-Id) \neq 0$ and satisfying equations $(M,c,i)$ for $1 \leq i \leq r$. We will look for a particular family of matrices $M$, that is,
\[
M=m^kN+Id.
\] Here, $k$ will be some suitable positive number, $N=(j_{\alpha \beta})_{\alpha \beta}$ will be some integer $r \times r$ matrix with determinant $1$ and $m^kN=(m^k j_{\alpha \beta})_{\alpha \beta}$ is the natural scalar product of a number by a matrix. The first thing to observe is that any such matrix $M$ satisfies all the equations $(M,c,i)$ for some big enough $k \geq 1$. Let us see that. It is easy to see that, for such $M$, the equations $(M,c,i)$ become exactly
\[
(p_1^{j_{1i}y_1}p_2^{j_{2i}y_2}\dots p_r^{j_{ri}y_r})^{m^k}\equiv 1\ \ \text{mod}\ m^c.\ \ \ \ \ (M,c,i)
\]

For us to use the previous lemma, the term inside the parenthesis in the above equation must be congruent to $1$ modulo $m$, so we claim this is true. Since $m$ divides each number $p_s^{y_s}-1$ ($1 \leq s \leq r$) by definition, we have $p_s^{y_s}=1\ \text{mod}\ m$, so by the multiplicative property of integer congruence,it is clear that $p_1^{j_{1i}y_1}p_2^{j_{2i}y_2}\dots p_r^{j_{ri}y_r}\equiv 1\ \ \text{mod}\ m$, which shows our claim. Now let $k=c-1$. By the above lemma we have $(p_1^{j_{1i}y_1}p_2^{j_{2i}y_2}...p_r^{j_{ri}y_r})^{m^k}=1\ \text{mod}\ m^c$, so for every $i$, equation $(M,c,i)$ is satisfied for such $M$.

It is then enough for us to find, for any $r \geq 2$, an integer matrix $N$ which makes $det(M)=1$ and $det(M-Id) \neq 0$. Since $M=m^kN+Id$, we have
\[
det(M-Id)=det(m^kN)=m^{rk}det(N),
\]so for $det(M-Id)$ to be non-zero it suffices us to have $det(N) \neq 0$. We claim therefore that, for any $r \geq 2$, there is a matrix $N_r$ such that $det(N_r)=1$ and $det(M_r)=det(m^kN_r+Id)=1$. For any $r\geq 2$, let
\[ N_r=\left[ \begin{array}{ccccccc}
1 & -(m^k+2) & m^k+1 & -(m^k+1) & \dots & (-1)^{r-4}(m^k+1) & (-1)^{r-3}(m^k+1) \\
1 & -(m^k+1) & m^k & -m^k & \dots & (-1)^{r-4}m^k & (-1)^{r-3}m^k \\
0 & 1 & 0 & 0 & \dots & 0 & 0 \\
0 & 0 & 1 & 0 & \dots & 0 & 0 \\
\vdots & \vdots & \vdots & \vdots & \ddots & \vdots & \vdots \\
0 & 0 & 0 & 0 & \dots & 1 & 0 
\end{array} \right].\]

By developing the determinant of $N_r$ using the last column, we get that $\det (N_r)=1$, since the two submatrices that appear are upper triangular with diagonal entries equal 1. Now, our task is to prove that $\det (M_r)=1$, where
\[ M_r=\left[ \begin{array}{ccccccc}
d & -m^k(m^k+2) & m^kd & -m^kd & \dots & (-1)^{r-4}m^kd & (-1)^{r-3}m^kd \\
m^k & -m^kd+1 & m^{2k} & -m^{2k} & \dots & (-1)^{r-4}m^{2k} & (-1)^{r-3}m^{2k} \\
0 & m^k & 1 & 0 & \dots & 0 & 0 \\
0 & 0 & m^k & 1 & \dots & 0 & 0 \\
\vdots & \vdots & \vdots & \vdots & \ddots & \vdots & \vdots \\
0 & 0 & 0 & 0 & \dots & 1 & 0 \\
0 & 0 & 0 & 0 & \dots & m^k & 1
\end{array}\right]\]
with $d=m^k+1$. We will prove this by induction. The case $r=2$ is verified by the calculation of the determinant of
\[ M_2=\left[ \begin{array}{cc}
m^k+1 & -m^k(m^k+2) \\
m^k & -m^k(m^k+1)+1 
\end{array}\right].\]

Now, for $r>2$, developing the determinant of $M_r$ by the last column gives us:
\[ \begin{array}{rl}
\det (M_r) & =(-1)^{r+1}(-1)^{r-3}m^kdm^{k(r-1)}+(-1)^{r+2}(-1)^{r-3}m^{2k}dm^{k(r-2)}+(-1)^{2r}\det (M_{r-1}) \\
& =m^{kr}d-m^{kr}d+1 \\
& =1.
\end{array}\]

This completes the induction step and finishes our proof.
\end{proof}


\end{document}